\documentclass[12pt]{amsart}

\usepackage{latexsym,amsfonts,amsmath,amssymb,amsthm,url,graphics,psfrag,amscd,stmaryrd}
\usepackage[english]{babel}
\usepackage[latin1]{inputenc}
\usepackage[T1]{fontenc}
\usepackage{graphics}
\usepackage{graphicx}
\usepackage{color}

\newcommand{\N}{\mathbb{N}}

\newcommand{\Z}{\mathbb{Z}}

\newcommand{\pp}{\mathbb{P}}

\newcommand{\kA}{\mathcal{A}}
\newcommand{\kB}{\mathcal{B}}
\newcommand{\kC}{\mathcal{C}}

\newcommand{\kO}{\mathcal{O}}

\newcommand{\kF}{\mathcal{F}}
\newcommand{\kG}{\mathcal{G}}

\newcommand{\kM}{\mathcal{M}}
\newcommand{\kE}{\mathcal{E}}
\newcommand{\kH}{\mathcal{H}}

\newcommand{\kN}{\mathcal{N}}

\newcommand{\kL}{\mathcal{L}}

\newcommand{\pn}{\mathbb{P} }

\newcommand{\en}{\mathbb{E} }

\newcommand{\cov}{\textrm{Cov}}
\newcommand{\bin}{\mathcal{B}}

\newcommand{\lin}{\left[\kern-0.15em\left[}
\newcommand{\rin} {\right]\kern-0.15em\right]}
\newcommand{\linf}{[\kern-0.15em [}
\newcommand{\rinf} {]\kern-0.15em ]}
\newcommand{\ilin}{\left]\kern-0.15em\left]}
\newcommand{\irin} {\right[\kern-0.15em\right[}

\newtheorem{lem}{Lemma}[section]
\newtheorem{rema}[lem]{Remark}
\newtheorem{prop}[lem]{Proposition}
\newtheorem{theo}[lem]{Theorem}

\title[Contact process on conf. model with infinite mean degree]
       {\bf Metastability for the contact process on the configuration model with infinite mean degree}

\author{Van Hao Can}
\author{Bruno Schapira}
\address{Aix Marseille Universit\'e, CNRS, Centrale Marseille, I2M, UMR 7373, 13453 Marseille, France}
\email{cvhao89@gmail.com}
\email{bruno.schapira@univ-amu.fr}


\begin{document}

\maketitle
\begin{abstract}
We study the contact process on the configuration model with a power law degree distribution, 
when the exponent is smaller than or equal to two. 
We prove that the extinction time grows exponentially fast with the size of the graph 
and prove two metastability results.  First the extinction time divided by its mean converges in distribution toward 
an exponential random variable with mean one, when the size of the graph tends to infinity. Moreover, the density of infected sites taken at exponential times converges in probability to a constant. This extends previous results in the case of an exponent larger than $2$ obtained in \cite{CD,MMVY,MVY}.  
\end{abstract}

\section{Introduction}
In this paper we will prove metastability results for the contact process on the configuration model 
with a power-law degree distribution, extending the main results of \cite{CD, MVY,MMVY} to the case when the exponent of the power-law is smaller than or equal to $2$. 

\vspace{0.2cm}
The contact process is one of the most studied interacting particle systems, see in particular Liggett's book \cite{L}, 
and is also often interpreted as a model for the spread of a virus in a population or a network.  
Mathematically, it can be defined as follows: given a countable locally finite graph $G$ and $\lambda >0$,  
the contact process on $G$ with infection rate $\lambda$ is a continuous-time Markov process $(\xi_t)_{t\geq 0}$ on $\{0,1\}^V$, 
with $V$ the vertex set of $G$. The elements of $V$, also called sites, are regarded as individuals which are either 
infected (state $1$) or healthy (state $0$). By considering $\xi_t$ as a subset of $V$ via $\xi_t \equiv \{v: \xi_t(v)=1\}$, 
the transition rates are given by
\begin{align*}
\xi_t \rightarrow \xi_t \setminus \{v\} & \textrm{ for $v \in \xi_t$ at rate $1,$ and } \\
\xi_t \rightarrow \xi_t \cup \{v\} & \textrm{ for $v \not \in \xi_t$ at rate }  \lambda \, \textrm{deg}_{\xi_t}(v),
\end{align*}
where $\textrm{deg}_{\xi_t}(v)$ denotes the number of edges between $v$ and another infected site  
(note that if $G$ is a simple graph, in the sense that there is only one edge between any pair of vertices, then $\textrm{deg}_{\xi_t}(v)$ is just the number of infected neighbors of $v$ at time $t$). 
 
\vspace{0.2cm}
Since the empty configuration is an absorbing state (and the only one), a quantity of particular interest is the extinction time, defined by 
$$\tau_G = \inf \{t: \xi_t = \varnothing\}.$$
Exploiting the fact that the contact process is stochastically increasing in $\lambda$, one can show that some graphs exhibit a 
nontrivial phase transition, regarding the finiteness of $\tau_G$. 
For instance on $\Z^d$, there exists a critical value $\lambda_c(d)>0$, such that for 
$\lambda\le \lambda_c(d)$, $\tau_{\Z^d}$ is a.s. finite (when $\xi_0$ is finite), 
whereas when $\lambda>\lambda_c(d)$, it has positive probability to be infinite (even when starting from a single vertex),  
see \cite{L} Section I.2 for a proof of this and references.

Here we will only consider finite graphs, in which case the extinction time is always almost surely finite. However, it is still 
interesting to understand its order of magnitude as a function of the size of the graph. 
For instance a striking phenomenon occurs on finite boxes 
$\llbracket 0,n \rrbracket^d$: one can show that with high probability (w.h.p.), if the process starts from full occupancy, the extinction time is of logarithmic order when 
$\lambda<\lambda_c(d)$, of polynomial order when $\lambda=\lambda_c(d)$ (at least in dimension one), 
and of exponential order when 
$\lambda>\lambda_c(d)$, see \cite{L} Section I.3 for a discussion on this and a complete list of references.

In fact such result seems intimately related to the fact that finite boxes converge to $\Z^d$ when $n$ tends to infinity, 
in the sense of the Benjamini--Schramm's local weak convergence of graphs \cite{BS}. 
If a rigorous connection between the two phenomena still remains conjectural at the moment, 
recently many examples gave substancial credit to this conjecture, see for instance  \cite{CD, CMMV, MV, MMVY}.   

\vspace{0.2cm} The case of the configuration model (a definition will be given later) is particularly interesting in this regard, 
at least when the degree distribution has finite mean. Indeed in this case it is not difficult to see 
that when the number of vertices increases, the sequence of graphs converges toward a Galton Watson tree.  
In \cite{CD} Chatterjee and Durret have shown that when the degree distribution has a power law (with exponent larger than two), 
the extinction time grows faster than any stretched exponential (in the number of vertices), 
which can be interpreted in saying that the critical value is zero for these graphs 
(invalidating thereby some physicists predictions). Since on the other hand one can show 
that the critical value on the limiting Galton Watson tree is also zero 
(the process has always a positive probability to survive for any $\lambda>0$), 
the conjecture mentioned above is satisfied for this class of examples. It is worth noting that the case of degree distributions 
with lighter tails than polynomial seems much harder 
(in particular understanding the case of Poisson distributions would be of great interest due to its connection with 
Erd\"os-R\'enyi random graphs).

But the configuration model is also interesting for another reason, highlighted in \cite{CD}: 
when the degree sequence has a power law, the contact process exhibits a metastable behaviour. 
This was first proved under a finite second moment hypothesis (equivalently for exponents larger than three) in \cite{CD}, and the result has been later strengthened 
and extended to exponents 
larger than two in \cite{MVY, MMVY}. To be more precise now, in \cite{CD} the authors proved that when the degree distribution has a power law with finite second moment, then 
 $$\pn \left( c\lambda^{1+(a-2)(2- \delta)} \leq \frac{|\xi_{\exp(\sqrt n)}|}{n} \leq C \lambda^{1+(a-2)(1- \delta)} \right) \rightarrow 1,$$
for some positive constants $c$ and $C$ (independent of $\lambda$), 
where $\xi$ denotes the contact process starting from full occupancy.    
In \cite{MMVY} the authors have shown that when the degree distribution has finite mean (and a power law), 
the extinction time is w.h.p. exponential in the size of the graph (when starting from full occupancy), 
and combined with the results of \cite{MVY}, 
one obtains that 
 $$ \mathbb{P} \left(   c \rho_{a}(\lambda) \leq \frac{|\xi_{t_n}|}{n} \leq  C \rho_{a}(\lambda)   \right) \rightarrow 1,$$
for any sequence $(t_n)$ satisfying $t_n \to \infty$ and $t_n \le \exp(cn)$, where  
\begin{displaymath}
\rho_{a}(\lambda) = \left \{ \begin{array}{ll}
\lambda^{\frac{1}{3-a}} & \textrm{ if } 2 < a \leq 5/2\\
\frac{\lambda ^{2a-3}}{\log ^{a-2} (\frac{1}{\lambda})} & \textrm{ if } 5/2 < a \leq 3\\
\frac{\lambda ^{2a-3}}{\log ^{2a-4} (\frac{1}{\lambda})} & \textrm{ if }  a > 3.
\end{array} \right.
\end{displaymath}

\noindent In this paper we complete this picture by studying the case of power laws with exponents  $ a\in (1,2]$. 
To simplify the discussion and some proofs we have chosen to consider mainly only two special choices of degree distribution. 
Namely we assume that it is given either by 
\begin{equation} \label{pnaj}
   p_{n,a}(j)=   c_{n,a}\, j^{-a}\qquad \textrm{for }j=1,\ldots,n, 
\end{equation}
for graphs of size $n$, or by 
\begin{equation}
\label{paj}
p_a(j) = c_{\infty,a}\, j^{-a} \qquad \textrm{for }j\ge 1,
\end{equation}
independently of the size of the graph, where $(c_{n,a})$ and $c_{\infty,a}$ are normalizing constants. 
However, at the end of the paper we also present straightforward extensions of our results to more general distributions, 
see Section \ref{secext} for more details. Our first main result in this setting is the following:

\begin{theo} \label{td}
For each $n$, let $G_n$ be the configuration model with $n$ vertices and degree distribution given either by \eqref{pnaj} or \eqref{paj} with $a\in (1,2]$. Consider the contact process $(\xi_t)_{t\ge 0}$ with infection rate $\lambda>0$ starting from full occupancy on $G_n$.
Then there is some positive constant $c = c(\lambda)$, such that the following convergence in probability holds:
\begin{equation} \label{etd}
\frac{|\xi_{t_n}|}{n}\quad  \mathop{\longrightarrow}^{(\pn)}_{n\to\infty} \quad    \rho_{a}(\lambda), 
\end{equation}
for any sequence $(t_n)$ sastifying  $ t_n  \to \infty$ and $t_n \leq \exp(c n)$, where
\begin{equation}
\rho_{a}(\lambda) = \sum \limits_{j=1}^{\infty} \frac{j\lambda}{j\lambda+1} p_a(j).
\end{equation}
\end{theo}
\noindent Note that as $\lambda \rightarrow 0$, 
\begin{displaymath}
\rho_{a}(\lambda)\asymp \left \{ \begin{array}{ll}
\lambda^{a-1} & \textrm{ if } 1 < a < 2 \\
\lambda \log \frac{1}{\lambda} & \textrm{ if } a =2,
\end{array} \right.
\end{displaymath}
which in particular shows that the guess of Chatterjee and Durrett \cite{CD} that $\rho_a(\lambda)$ should be   
$\kO(\lambda)$ was not correct. 

\vspace{0.2cm}
Now let us make some comments on the proof of this result. One first remark is that one of the main ingredients 
in the approach of \cite{MVY} completely breaks down when the degree distribution has infinite mean 
(or when its mean is unbounded like in the case \eqref{pnaj}), since in this case 
the sequence of graphs $(G_n)$ does not locally converge anymore. 
In particular we cannot transpose the analysis of the contact process on $G_n$ (starting from a single vertex) 
into an analysis on an infinite limit graph. So instead we have to work directly on the graph $G_n$. 
In fact we will show that it contains w.h.p. a certain number of disjoint star graphs 
(i.e. graphs with one central vertex and all the others connected to the central vertex), 
which are all connected, and whose total size is of order $n$ (the size of $G_n$). 
It is well known that the contact process on a star graph remains active w.h.p. for a time exponential in the size of the graph. 
So our main contribution here is to show that when we connect disjoint star graphs together, 
the process survives w.h.p. for a time which is exponential in the total size of these graphs.  
To this end we use the machinery introduced in \cite{CD}, with their notion of lit stars. We refer to  
Proposition \ref{psta} and its proof for more details. 
Now it is interesting to notice that while this strategy works in all the cases we consider, 
the details of the arguments strongly depend on whether $a<2$ or $a=2$, and on the choice of the degree distribution. 
This explains why we found interesting to present the proof for the two examples \eqref{pnaj} and \eqref{paj}  
(note that these distributions were also considered in \cite{VVHZ}, 
where it was already proved that the distance between two randomly chosen vertices was a.s. equal either to two or three).

Then to obtain the asymptotic expression for the density \eqref{etd}, the point is to use the 
self-duality of the contact process. This allows to transpose the problem on the density of infected sites 
in terms of survival of the process starting from a single vertex. But starting from a single vertex, 
the process has a real chance to survive for a long time only if it infects one of its neighbors before extinction. 
Moreover, when it does, one can show that w.h.p. it immediately infects one of the star graphs mentioned above, 
and therefore the virus survives w.h.p. for a time at least $t_n$. 
The conclusion of the theorem follows once we observe 
that the probability to infect a neighbor before extinction starting from any vertex is exactly equal to $\rho_a(\lambda)$ in case \eqref{paj} and to   
\begin{equation} \label{lm}
\rho_{n,a}(\lambda) := \sum \limits _{j=1}^n \frac{j \lambda}{j \lambda +1} p_{n,a}(j), 
\end{equation}
in case \eqref{pnaj}, which converges to $ \rho_a(\lambda)$, as $n \rightarrow \infty$.

\vspace{0.2cm}
\noindent Our second result is often considered in the literature as another (weaker) expression of the metastability:  

\begin{theo}
\label{propexp}
Assume that the degree distribution on $G_n$ is given either by \eqref{pnaj} or \eqref{paj} with $a\in (1,2]$, 
and let $\tau_n$ be the extinction time of the contact process with infection rate $\lambda>0$ starting from full occupancy. Then 
\begin{itemize} 
\item[(i)] the following convergence in law holds
$$\frac{\tau_n}{\en(\tau_n)}\ \mathop{\longrightarrow}^{(\mathcal L)}_{n\to \infty}  \  \kE(1),$$
with $\kE(1)$ an exponential random variable with mean one,  \\ 
\item[(ii)] there exists a constant $C>0$, such that $\en(\tau_n) \le \exp(Cn)$, for all $n\ge 1$. 
\end{itemize}
\end{theo}
In particular this result shows that Theorem \ref{td} cannot be extended to sequences $(t_n)$ growing faster than exponentially. 
In fact one can prove (see Remark \ref{tn}) that Theorem \ref{td} holds true for any constant $c$ smaller  
than $\liminf (1/n) \log \en(\tau_n)$, and cannot be extended above this limit. This of course raises the question  
of knowing if the sequence $(1/n)\log \en(\tau_n)$ admits a limit or not. Such result has been obtained in a number 
of contexts, for instance in \cite{MMVY} or on finite boxes $\llbracket 0,n\rrbracket^d$ (see \cite{L} Section I.3), 
but we could not obtain it in our setting. One reason, which for instance prevents us to apply the strategy of \cite{MMVY}, 
is that there does not seem to be a natural way to embed $G_n$ into $G_{n+1}$ (or another configuration model with larger size).  

\vspace{0.2cm}
Our method for proving Theorem \ref{propexp} (i) is rather general and only requires some simple hypothesis on the maximal degree and the diameter of the graph, which is satisfied in most scale-free random graphs models, like 
the configuration model with power law distribution 
having a finite mean (with the same hypothesis as in \cite{CD,MVY}), or the preferential attachment graph (see \cite{C}).  
We refer the reader to Proposition \ref{pcel} and Remark \ref{remexp} for more details.

\vspace{0.2cm} Let us also stress the fact that (ii) would be well known if the graph had order $n$ edges, as when 
the degrees have finite mean, but here it is not the case, so we have to use a more specific argument, see Section 6.

\vspace{0.2cm}
Now the paper is organized as follows. In the next section, we recall the well-known and very usefull graphical construction of 
the contact process. We also give a definition of the configuration model, fix some notation, and prove preliminary  
results on the graph structure. In Section 3, we prove that $G_n$ contains w.h.p. a subgraph, called two-step star graph, 
which is made of several star graphs connected together, whose total size is comparable to the size of the whole graph. 
We refer to this section for a precise statement, which in fact depends on which case we consider ($a<2$ or $a=2$,  
and distribution \eqref{pnaj} or \eqref{paj}). 
In Section 4 we show that once a vertex (with high degree) of the two-step star graph is infected, the virus survives 
for an exponential time. Then we prove Theorem \ref{td} and \ref{propexp} in Sections 5 and 6 respectively. 
Finally in the last section we discuss several extenstions of our results to more general degree distributions.

\section{Preliminaries}
\subsection{Graphical construction of the contact process.}
We briefly recall here the graphical construction of the contact process (see more in Liggett's book \cite{L}).

Fix $\lambda>0$ and an oriented graph $G$ (recall that a non-oriented graph can also be seen as oriented by associating 
to each edge two oriented edges). Then assign independent Poisson point processes $\mathcal{N}_v$ of rate $1$ to each 
vertex $v \in V$ and $\mathcal{N}_{e}$ of rate $\lambda$ to each  oriented edge $e$. 
Set also $\mathcal{N}_{(v,w)}:=\cup_{e : v\to w}\, \mathcal{N}_e$,  
for each ordered pair $(v,w)$ of vertices, where the notation $e: v\to w$ means that the oriented edge $e$ goes from $v$ to $w$.

We say that there is an infection path from $(v,s)$ to $(w,t)$, and we denote it by 
\begin{eqnarray}
\label{vswt}
(v,s)\longleftrightarrow (w,t), 
\end{eqnarray}
either if $s=t$ and $v=w$, or if $s<t$ and if there is a sequence of times $s=s_0< s_1<\ldots<s_l<s_{l+1}=t,$ and a sequence of vertices $v=v_0,v_1,\ldots,v_l=w$ such that for every $i=1,\ldots,l$ 
\begin{displaymath}
\left \{ \begin{array}{ll}
 s_i \in \mathcal{N}_{(v_{i-1},v_i)} \quad \textrm{ and }\\
 \mathcal{N}_{v_i} \cap [s_i, s_{i+1}] = \varnothing.
\end{array} \right.
\end{displaymath}
Furthermore, for any $A$, $B$ two subsets of $V_n$ and $I$, $J$ two subsets of $[0,\infty)$, we write 
$$A\times I \longleftrightarrow B\times J,$$
if there exists $v\in A$, $w\in B$, $s\in I$ and $t\in J$, such that \eqref{vswt} holds. 
Then for any $A\subset V_n$, the contact process with initial configuration $A$ is defined by  
$$\xi^A_t :=\left\{v \in V_n:  A\times\{0\}\longleftrightarrow (v,t)\right\},$$
for all $t\ge 0$.  
It is well known that $(\xi^A_t)_{t \geq 0}$ has the same distribution as the process defined in the introduction.
Just note that in our definition, the Poisson processes associated to edges forming loops play no role 
(we could in particular remove them), but this definition will be convenient at one place of the proof (when we will use that the $Y_{n,v}$'s are i.i.d. in Subsection \ref{subsectionYnv}). 
We define next $\tau_n^A$ as the extinction time of the contact process starting from $A$. 
However, we will sometimes drop the superscript $A$ from the notation when it will be clear from the context. We 
will also simply write $\xi^v_t$ or $\tau_n^v$ when $A=\{v\}$. 

\vspace{0.3cm}
\noindent Finally we introduce the following related notation:
\begin{equation} \label{sm}
\sigma(v)= \inf \{s \geq 0: s \in \mathcal{N}_v \},
\end{equation}
and 
\begin{equation} \label{sme}
\sigma(e)= \inf \{s \geq 0: s \in \mathcal{N}_e \},
\end{equation}
for any vertex $v$ and oriented edge $e$. 

\subsection{Configuration model and notation.}
The configuration model is a well known model of random graph with prescribed degree distribution, see for instance \cite{V}. 
In fact here we will consider a sequence $(G_n)$ of such graphs. 
To define it, start for each $n$ with a vertex set $V_n$ of cardinality 
$n$ and construct the edge set as follows. Consider a sequence of i.i.d. integer valued random variables $(D_v)_{v\in V_n}$
(whose law might depend on $n$) and assume that $L_n =\sum_v D_v$ is even (if not increase one of the $D_v$'s by $1$, 
which makes no difference in what follows). For each vertex $v$, start with $D_v$ half-edges (sometimes called stubs) 
incident to $v$. Then match uniformly at random all these stubs by pairs. Once paired two stubs form an edge of the graph. 
Note that the random graph we obtain may contain multiple edges (i.e. edges between the same two vertices), 
or loops (edges whose two extremities are the same vertex). 

In fact one can also define $G_n$ by matching the stubs sequentially. This equivalent construction will be used in particular in Lemma \ref{lta1}, \ref{2a<2} and  \ref{ltb1}, so let us describe it now. As with the previous construction we start with a sequence of degrees $(D_v)_{v\in V_n}$, and for each $v\in V_n$, $D_v$ half-edges emanating from $v$. Then we denote by $\kH$ the set of all the half-edges. Select one of them $h_1$ arbitrarily and then choose a half-edge $h_2$ uniformly from $\kH \setminus \{h_1\}$, and match $h_1$ and $h_2$ to form an edge. Next, select arbitrarily another half-edge $h_3$ from $\kH \setminus \{h_1, h_2\}$  and match it to another $h_4$ uniformly chosen from  $\kH \setminus \{h_1, h_2, h_3\}$. Then continue this procedure until there are no more half-edges. It is possible to show that the two constructions of $G_n$ have the same law.

\vspace{0,2cm}
Now we introduce some notation. We denote the indicator function of a set $E$ by ${\bf 1}(E)$.  
For any vertices $v$ and $w$ we write $v \sim w$ if there is an edge between them (in which case we say that they are neighbors or connected), and $v \not \sim w$ otherwise. We also denote by $s_v$ the number of half-edges forming loops attached to a vertex $v$. We call size of a graph $G$ the cardinality of its set of vertices, and we denote it by $|G|$. 

\vspace{0.2cm}
A graph in which  all vertices have degree one, except one which is connected to all the others is called a {\bf star graph}. The only vertex with degree larger than one is called the center of the star graph, or central vertex. 
We call {\bf two-step star graph} a graph formed by a family of disjoints star graphs, denoted by $S(v_i)_{1\le i\le k}$,  
centered respectively in vertices $(v_i)_{1\le i\le k}$, plus an additional vertex $v_0$ and edges between $v_0$  
and all the $v_i'$'s 
(or equivalently it is just a tree, which is of height $2$ when rooted at $v_0$). The notation ${\bf S(k; d_1,\dots,d_k)}$ 
will refer to the two-step star graph where $v_i$ has degree $d_i+1$ for all $i$ (which means that inside $S(v_i)$, 
$v_i$ has degree $d_i$, or that $S(v_i)$ has size $d_i+1$). These graphs will play a crucial role in our proof of Theorem \ref{td}.

\vspace{0.2cm}
Furthermore we denote by $\bin(n,p)$ the binomial distribution with parameters $n$ and $p$. 
If $f $ and $g$ are two real functions, we write $f= \mathcal{O}(g)$ if there exists a constant $C>0,$ such that $f(x) \leq C g(x)$ for all $x ;$  $f \asymp g $ if $f= \mathcal{O}(g)$ and $g= \mathcal{O}(f);$  $f=o(g)$ if $g(x)/f(x) \rightarrow 0$ as $x \rightarrow \infty$. 
Finally  for a sequence of random variables $(X_n)$ and a function $f:\N \to (0,\infty)$, we say that $X_n \asymp f(n)$ holds w.h.p. if there exist positive constants $c$ and $C,$ such that $\pn(c f(n) \leq X_n \leq C f(n)) \rightarrow 1$, as $n\to \infty$.

\subsection{Preliminary estimates on the graph structure}
We first recall a large deviations result which we will use throughout this paper (see for instance \cite{DZ}): if $X \sim \bin(n,p),$ then for all $c>0$, there exists $\theta>0$, such that
\begin{equation} \label{ld}
\pn(|X-np| \geq cnp) \leq \exp(-\theta np ) \quad \textrm{   for all } n \in \mathbb{N} \textrm{ and } p \in [0,1].  
\end{equation}

\noindent Now we present a series of lemmas deriving basic estimates on the degree sequence and the graph structure. The first one is very elementary and applies to all the cases we will consider in this paper.  
\begin{lem} \label{lb}
Assume that the degree sequence is given either by \eqref{pnaj} or \eqref{paj}, with $1< a \leq 2$. For $j\ge 1$, let $A_j:=\{v : D_v =j\}$ and  $n_j= |A_j|$. 
Then there exist positive constants $c$ and $C$, such that 
$$ \pn (n_j \in (c n j^{-a},C n j^{-a}) \textrm{ for all }j=1,...,n^{1/2a})  = 1- o(1).$$
\end{lem}
\begin{proof}
Observe that we always have $ n_j \sim \bin (n,p_j)$, for some $p_j\in  (c_{\infty,a} j^{-a}, j^{-a})$, with $c_{\infty,a}$ as in \eqref{paj}. Thus the result directly follows from \eqref{ld}.  
\end{proof}

\noindent
Our next results depend more substantially on the value of $a$ and the choice of the degree distribution.  
\begin{lem} \label{lta1}
Assume that the degree distribution is given by \eqref{pnaj}, with $a \in(1,2)$. 
Let $E := \{ v\, :\, D_v \geq n/2 \}$. Let also $\kappa > 2-a$ and $\chi<1$ be some constants. 
Then the following assertions hold   
\begin{itemize}
\item[(i)] $ L_n \asymp n^{3-a} $ w.h.p., \\
\item[(ii)] $ |E| \asymp n^{2-a} $ w.h.p., \\
\item[(iii)] $\pn( v\sim w \textrm{ for all $v$ and $w$ such that $D_v \geq n/2$ and $D_w \geq n^{\kappa}$}) = 1-o(1)$, \\
\item[(iv)] $\pn(s_v\ge 1) = o(1)$, for any $v\in V_n$, \\
\item[(v)]  $\pn\left(\textrm{All neighbors of $v$ have degree larger than $n^\chi$} \right) = 1-o(1)$, for any $v\in V_n$.
\end{itemize}
\end{lem} 
\begin{proof} Let us start with Part (i). It follows from the definition \eqref{pnaj} that 
$$\mathbb{E} \left( D_v\right) \asymp n^{2-a} \quad \textrm{and} \quad \textrm{Var}(D_v) \asymp n^{3-a}. $$
The result follows by using Chebyshev's inequality.

Part (ii) is similar to Lemma \ref{lb}. 
For Part (iii), let $v$ and $w$ be two vertices such that $D_v \geq n/2 $ and $ D_w \geq n^{\kappa}$. Then conditionally on $(D_z)_{z \in V_n}$, 
the probability that the $n/8$ first stubs of $v$ do not connect to $w$ is smaller than $(1- \frac{n^{\kappa}}{L_n-n/4})^{n/8}$. 
Hence,
$$ \pn \left(v \not \sim w \mid (D_z)_{z \in V_n}, \, L_n \in (c n^{3-a},C n^{3-a})\right) \leq \left(1- \frac{n^{\kappa}}{Cn^{3-a}-n/4} \right)^{n/8} = o(n^{-2}),
$$
which proves (iii) by using (i) and a union bound.

We now prove (iv). To this end, notice that conditionally to $D_v$ and $L_n$, $s_v$ 
is stochastically dominated by a binomial random variable with parameters $D_v$ and $D_v/(L_n-2D_v+2)$
(remark in particular that since $D_z\ge 1$ for all $z$, the denominator in the last term is always positive). 
Hence Markov's inequality shows that 
$$\pn(s_v\ge 1\mid D_v,L_n) \le \frac{D_v^2}{L_n-2D_v+2}.$$
The result follows by using (i) and that for any fixed $\varepsilon>0$, $\pn(D_v\ge n^\varepsilon) = o(1)$.

It remains to prove (v). Denote the degrees of the neighbors of $v$ by $D_{v,i}$, $i = 1,\ldots,D_v$. 
It follows from the definition of the configuration model that for any $i\le D_v$ and $k\neq D_v$,  
$$\pn(D_{v,i} =k \mid (D_z)_{z\in V_n} )= \frac{kn_k}{L_n-1},$$  
where we recall that $n_k$ is the number of vertices of degree $k$. 
Therefore, 
\begin{align*}
\pn(D_{v,i} \le n^\chi \mid (D_z)_{z\in V_n} ) \le \frac{K_n}{L_n-1},
\end{align*}
where 
$$K_n = \sum_{k\le n^\chi} k n_k.$$
Summing over $i$, we get 
$$\pn(\exists i\le D_v : D_{v,i} \le n^\chi \mid (D_z)_{z\in V_n} ) \le \frac{K_nD_v}{L_n-1}.$$
Moreover, similarly to the proof of (i), we can see that 
w.h.p. 
$$K_n\asymp n ^{1+\chi(2-a)}.$$
Together with (i), and using again that $D_v\le n^\varepsilon$ w.h.p. for any fixed $\varepsilon>0$, 
we get (v). 
\end{proof}

\vspace{0.3cm}
\noindent Things drastically change when the degree distribution is given by \eqref{paj}. In this case $L_n$, as well as the 
$k$ maximal degrees, for any fixed $k$, are all of order $n^{1/(a-1)}$ (for the comparison with the previous case 
note that $1/(a-1)$ is always larger than $a-3$ when $a\in(1,2)$, which is consistent with the fact that the distribution 
\eqref{paj} stochastically dominates \eqref{pnaj}): 

\begin{lem} \label{2a<2}
Assume that the degree distribution is given by \eqref{paj}, with $a\in(1,2)$. 
Denote by $(D_i)_{1\le i\le n}$ the sequence of degrees ranged in decreasing order 
(in particular $D_1$ is the maximal degree). 
Let also $\kappa > (2-a)/(a-1)$ and $\chi<1/(a-1)$ be some constants. 
Then the following assertions hold   
\begin{itemize}
\item[(i)] there exist (a.s. positive and finite) random variables $(\gamma_i)_{i\ge 0}$, 
such that for any fixed $k\ge 1$, 
\begin{equation*}
\left( \frac{L_n}{n^{1/(a-1)}}, \frac{D_1}{n^{1/(a-1)}},\ldots,\frac{D_k}{n^{1/(a-1)}} \right)\ \mathop{\longrightarrow}^{(\kL)}_{n\to \infty} \  (\gamma_0, \gamma_1,\ldots,\gamma_k).
\end{equation*} 
\item[(ii)] For any $\varepsilon >0$, there exists a positive constant $\eta=\eta(\varepsilon)$, such that for any fixed $k\ge 1$, 
$$\liminf_{n\to \infty} \, \pn\left(D_i/L_n \ge \eta\, i^{-1/(a-1)} \quad  \textrm{for all } 1\le i\le k\right) \ge 1-\varepsilon,$$
and an integer $k=k(\varepsilon)$, such that 
$$\liminf_{n\to \infty} \, \pn(D_1+\dots +D_k\ge L_n/2) \ge 1-\varepsilon.$$ 
\item[(iii)] $\pn( v\sim w \textrm{ for all $v$ and $w$ such that $D_v \geq n$ and $D_w \geq n^{\kappa}$}) = 1-o(1)$, \\
\item[(iv)] $\pn(s_v\ge 1) = o(1)$, for any $v\in V_n$, \\
\item[(v)]  $\pn\left(\textrm{All neighbors of $v$ have degree larger than $n^\chi$} \right) = 1-o(1)$, for any $v\in V_n$.
\end{itemize}
\end{lem} 
\begin{proof}
Part (i) is standard, we refer for instance to Lemma 2.1 in \cite{VVHZ}. More precisely   
let $(e_i)_{i \geq 1} $  be an i.i.d. sequence of exponential random variables with mean one and $\Gamma_i= e_1+\ldots+e_i$, for all $i\ge 1$ (in particular $\Gamma_i$ is a Gamma random variable with parameters $i$ and $1$). Then the result holds with 
$$\gamma_i = ((a-1)\Gamma_i/c_{\infty,a})^{-1/(a-1)},$$ 
for all $i\ge 1$, and $\gamma_0=\sum_i \gamma_i$ (which is well a.s. a convergent series).

For (ii) note that $\Gamma_i/i\to 1$ a.s. as $i\to \infty$. In particular for any $\varepsilon$, there exists $C>0$, such that 
$$\pn(\Gamma_i\le Ci \textrm{ for all }i\ge 1)\ge 1-\varepsilon/2.$$
The first assertion follows with (i), using also that $\pn(\gamma_0\le C)\ge 1-\varepsilon/2$, for $C$ large enough. The second one is an immediate corollary of (i) and the definition of $\gamma_0$ as the limit of the partial sum $\sum_{i\le k} \gamma_i$, as $k\to \infty$.

Parts (iii)-(v) are similar to the previous case. 
\end{proof}

\noindent
We now give an analogous result for the case $a=2$, which we will not prove here since it is entirely similar to the case $a<2$ (just for the case when the degree distribution is given by \eqref{paj}, one can use the elementary fact that w.h.p. all vertices have degree smaller than $n\log \log n$). 
\begin{lem} \label{ltb1} Assume that the degree distribution is given either by \eqref{pnaj} or \eqref{paj}, with $a=2$. 
Let $E': =\{v :  D_v \geq n^{3/4}\}$. Then the following assertions hold
\begin{itemize}
\item[(i)] $L_n \asymp n \log n$ w.h.p.,   \\
\item[(ii)] $|E'|\asymp n^{1/4} \quad  \textrm{and}\quad  \sum_{v\in E'} D_v \asymp n\log n \quad  \textrm{w.h.p.}$, \\ 
\item[(iii)] $\pn(v\sim w \textrm{ for all $v$ and $w$ such that }  D_v \geq n/\log n \textrm{ and } D_w \geq (\log n)^4) =1- o(1)$,\\ 
\item[(iv)] $\pn(s_v\ge 1) = o(1)$, for any $v\in V_n$.\\
\item[(v)] $\pn\left(\textrm{All neighbors of $v$ have degree larger than $(\log n)^4$} \right) = 1-o(1)$, for any $v\in V_n$.
\end{itemize}
\end{lem}

\section{Existence of a large two-step star graph}
In this section we will prove that the graph $G_n$ contains w.h.p. a large two-step star graph 
$S(k;d_1,\dots,d_k)$, the term large meaning that $d_1+\dots +d_k$ will be of order $n$, and all the $d_i$'s of order at least $\log n$. 
However, the precise values of $k$ and the $d_i$'s 
will depend on which case we consider 
(to be more precise, in the case of degree distribution given by \eqref{paj} 
with $a\in(1,2)$ we prove that for any $\varepsilon>0$, $G_n$ contains a large two-step star graph with probability 
at least $1-\varepsilon$, with $k$ and the $d_i$'s depending on $\varepsilon$. Nevertheless, 
the rest of the proof works mutadis mutandis).  

\subsection{Case $1<a<2$}
\subsubsection{Bounded degree sequence}
We assume here that the law of the degrees is given by \eqref{pnaj}.
Recall that $E= \{v: D_v \geq n/2\}$ and $A_1 =\{v : D_v =1\}$. 
In addition for any vertex $v$, let us denote by
$$d_1(v):= \sum_{w \in A_1} {\bf 1}(\{w \sim v\}),$$
the number of neighbors of $v$ in $A_1$. 

\begin{lem} \label{q}
There exist  positive contants $\beta $ and $\kappa$, such that 
\begin{equation} \label{eta2}
\pn \left( \# \{v \in E: d_1(v) \geq \beta n^{a-1}\} \geq \kappa\,  n^{2-a}\right) = 1-o(1).
\end{equation}
\end{lem}
\begin{proof}
It follows from the definition of the configuration model that for any $w \in A_1$ and $v \in E$,
\begin{align}
\label{wv}
\pn(w \sim v \mid (D_z)_{z\in V_n}) & = \frac{D_v}{L_n-1}.
\end{align}
Similarly for any $v\in E$ and $w \neq w' \in A_1$, 
\begin{align} \label{c1}
|\cov(w \sim v, w' \sim v\mid (D_z)) |& =  \left| \frac{D_v(D_v-1)}{(L_n-1)(L_n-3)} -\left( \frac{D_v}{L_n-1}\right)^2\right| \notag \\
& = \mathcal O\left( \frac{D_v}{L_n^2}\right) 
\end{align}
Define now the set 
$$\kA_n := \left\{cn^{3-a}\le L_n\le Cn^{3-a}\right\}\cap \left\{|A_1|\ge c n\right\},$$
with $0<c\le C$, such that 
\begin{equation}
\label{Gn}
\pp(\kA_n)= 1-o(1). 
\end{equation}
Note that the existence of $c$ and $C$ is guaranteed by Lemma \ref{lb} and \ref{lta1}. 
Set also $\beta=c/(4C)$. 
Then \eqref{wv} and \eqref{c1} show that on $\mathcal{A}_n$, 
\begin{equation*}
\label{sumwv}
\sum\limits_{w \in A_1} \pn (w \sim v \mid (D_z)) \geq 2\beta n^{a-1}, 
\end{equation*}
and 
$$\sum\limits_{w \neq w'\in A_1} \cov(w \sim v,w'\sim v \mid (D_z)) =o(n^{2a-2}).$$
Thus by using Chebyshev's inequality, we deduce that on $\mathcal{A}_n$, 
\begin{equation*}
\pn(d_1(v) \geq \beta n^{a-1} \mid (D_z))  = \pn \left(\sum\limits_{w \in A_1} {\bf 1}( \{w \sim v \}) \geq \beta n^{a-1} \ \Big| \ (D_z)\right) = 1- o(1).  
\end{equation*}
Hence for any $v \neq w \in E$,
\begin{align*}
\cov(d_1(v) \geq \beta n^{a-1}, d_1(w) \geq \beta n^{a-1}\mid (D_z)_{z\in V_n}) =o(1).
\end{align*}
Then by using Chebyshev's inequality again we obtain that on the event $\{|E|\ge 2\kappa n^{2-a}\}$, 
\begin{equation*} 
\pn \left( \# \{v \in E: d_1(v) \geq \beta n^{a-1}\} \geq \kappa\,  n^{2-a}\mid (D_z)_{z\in V_n}\right) = 1-o(1).
\end{equation*}
Then  \eqref{eta2} follows by using \eqref{Gn}, Lemma \ref{lta1} (ii) and  taking expectation. 
\end{proof}

\noindent As a corollary we get the following result: 

\begin{prop}
\label{stara<2}
Assume that the law of the degree sequence is given by \eqref{pnaj} with $a\in(1,2)$. 
There exist positive contants $\beta $ and $\kappa$, such that w.h.p. $G_n$ contains 
as a subgraph a copy of $S(k;d_1,\dots, d_k)$, with $k= \kappa n^{2-a}$ and $d_i= \beta n^{a-1}$, for all $i\le k$. 
\end{prop}  
\begin{proof}
This is a direct consequence of Lemma  \ref{lta1} (iii) and Lemma \ref{q}.   
\end{proof} 

\subsubsection{Unbounded degree sequences}
We assume here that the law of the degrees is given by \eqref{paj}. 
The proof of the next result is similar to the one of Lemma \ref{q}, so we omit it. 
\begin{lem}
\label{viri}
With the notation of Lemma \ref{2a<2}, let $(v_i)_{i\le n}$ be a reordering of the vertices of $G_n$, such that the degree of $v_i$ is $D_i$ for all $i$ (in particular $v_1$ is a vertex with maximal degree).  Then for any fixed $i$, 
$$\pn(d_1(v_i)\ge D_i n_1/(2L_n)) = 1-o(1).$$
\end{lem}
\noindent As a consequence we get 
\begin{prop}
Assume that the degree distribution is given by \eqref{paj}, with $a\in(1,2)$. There exists a constant $c>0$, such that 
for any $\varepsilon>0$, there exists $\eta=\eta(\varepsilon)>0$ and an integer $k=k(\varepsilon)$, such that for $n$ large enough, 
with probability at least $1-\varepsilon$, $G_n$ contains as a subgraph a copy of $S(k;d_1,\dots,d_k)$, 
with $d_i\ge \eta i^{-1/(a-1)}n$ for all $i\ge 1$, and $d_1+\dots +d_k\ge cn$. 
\end{prop}
\begin{proof}
It follows from Lemma \ref{viri} that for any $i$,
\begin{align*}
\pn(d_1(v_i)\ge D_i n_1/(2L_n)) = 1-o(1).
\end{align*}
Hence for any fixed $k$,
\begin{align*}
\pn \left( d_1(v_1) + \ldots + d_1(v_k)\ge \frac{n_1(D_1+ \ldots+ D_k)} { 2 L_n}  \right)= 1-o(1).
\end{align*}
Moreover, by Lemma \ref{lb} we have $\pn(n_1 \in (cn, Cn)) =1-o(1)$. On the other hand, for any $\varepsilon >0$, by Lemma \ref{2a<2} (ii), there exist $\eta= \eta(\varepsilon)$ and $k=k(\varepsilon)$, such that
\begin{align*}
\pn \left( \frac{D_1+ \ldots + D_k}{L_n } \geq \frac{1}{2} \right) &  \geq 1- \varepsilon/4, \\ 
 \pn(D_i/ L_n \geq \eta i^{-1/(a-1)} \quad \forall \, i \leq k+1)  & \geq 1- \varepsilon/4.
\end{align*}
 Therefore with probability at least $1- (3\varepsilon/4)$, for $n$ large enough, $\eta$ and $k$ as above,
 \begin{align*}
d_1(v_1) + \ldots + d_1(v_k) \geq cn/4, \\
 d_1(v_i) \geq c \eta i^{-1/(a-1)} n/2 \, \, \forall \,  i \leq k+1.
 \end{align*}
Then by using a similar argument as in the proof of Lemma \ref{lta1} (iii), we can show that with probability larger than $1- (\varepsilon/4)$, $v_{k+1}$ and $v_i$ are connected for all $i \leq k$. The result follows. 
\end{proof}

\subsection{Case $a=2$}
In this case we can treat both distributions \eqref{pnaj} and \eqref{paj} in the same way. 
Recall that $E'=\{v :  D_v \geq n^{3/4}\}$, and that $d_1(v)$ denotes the number of neighbors in $A_1$ of a vertex $v$. 

\begin{lem} \label{ex4}
There exists a positive constant $\beta$, such that
\begin{equation} \label{e17}
\pn\left(d_1(v) \ge \beta D_v/\log n \quad \textrm{for all }v\in E'\right) =1-o(1).   
\end{equation}
\end{lem}

\begin{proof}
The proof is very close to the proof of Lemma \ref{q}. 
First, for any $v\in E'$ and $w\in A_1$, we have  
$$ \pn(w  \sim v\mid (D_z)) \asymp  \frac{ D_v}{L_n} ,$$
and furthermore for any $w\neq w'\in A_1$, 
$$|\cov(w \sim v, w' \sim v\mid (D_z)) | = \kO\left(\frac{D_v}{L_n^2}  \right).$$
Then by using Chebyshev's inequality, we get that for any $v\in E'$, 
 $$ \pn \left(d_1(v) \leq \beta D_v n_1/ L_n \mid (D_z) \right) =\kO\left(\frac{L_n}{n_1D_v}\right),$$
for some constant $\beta>0$. The desired result follows by using a union bound and then Lemma \ref{lb} and \ref{ltb1} (i)-(ii).  
\end{proof}
\noindent As a consequence we get 
\begin{prop}
\label{stara2}
Assume that the law of the degree distribution is given either by \eqref{pnaj} or \eqref{paj} and that $a=2$. 
There exists a positive constant $\beta$ such that w.h.p. $G_n$ contains 
as a subgraph a copy of $S(k;d_1,\dots, d_k)$, with $k\asymp n^{1/4}$, $d_i\ge \beta n^{3/4}/\log n$ for all $i\le k$, and $d_1+\dots + d_k \asymp n$.  
\end{prop}  
\begin{proof}
Just take for the $v_i$'s the elements of $E'$. Then use Lemma  \ref{ltb1} (ii)-(iii) and Lemma \ref{ex4}.   
\end{proof}

\section{Contact process on a two-step star graph}
In this section we will study the contact process on a two-step star graph. 
Our main result is the following:

\begin{prop} \label{psta}
There exist positive constants $c$ and $C$, such that for any two-step star graph $G=S(k;d_1,\dots,d_k)$, satisfying $d_i \geq C \log n/\lambda^2$, for all $i\le k$, and $d_1+...+d_k = n$, 
\begin{align*}
\pn\left(\tau_n^{v_1} \geq \exp(c  \lambda^2 n)\right) = 1- o(1),
\end{align*}
where $\tau_n^{v_1}$ is the extinction time of the contact process with infection parameter $\lambda\le 1$ starting from $v_1$ on $S(k;d_1,\dots,d_k)$. 
\end{prop}
Note that since we are only concerned with the extinction time here, there is no restriction in assuming $\lambda \le 1$, 
as the contact process is stochastically monotone in $\lambda$ (see \cite{L}). So when $\lambda>1$ the same result holds; 
one just has to remove the $\lambda$ everywhere in the statement of the proposition. 

\vspace{0.2cm} 
Now of course an important step in the proof is to understand the behavior of the process 
on a single star graph. This has already been studied for a long time, for instance it appears in Pemantle \cite{P}, 
and later in \cite{BBCS, CD, MVY}. We will collect all the results we need in Lemma \ref{lst} below, but before 
that we give some new definition. We say that a vertex $v$ is {\bf lit} (the term is taken from \cite{CD}) at some time $t$ if the proportion of its infected neighbors 
at time $t$ is larger than $\lambda/(16e)$ (note that in \cite{MMVY} the authors also use the term {\it infested} 
for a similar notion). 

\begin{lem} \label{lst}
There exists a constant $c\in (0,1)$, such that if $(\xi_t)$ is the contact process with parameter $\lambda\le 1$ 
on a star graph $S$ with center $v$, satisfying $\lambda^2 |S|\geq  64 e^2$, then 
\vspace{0.15cm}
\begin{itemize}
\item[(i)] 
$\pn(\xi_{\exp(c \lambda^2 |S|)} \neq \varnothing \mid v \textrm{ is lit at time } 0) \geq 1-\exp(-c \lambda^2 |S|)$, \\ 
\item[(ii)] 
$\pn(\exists t>0 : v \textrm{ is lit at time }t\mid \xi_0(v)=1)\to 1\qquad \textrm{as }|S|\to \infty$.  \\
\item[(iii)] 
$\pn( v \textrm{ is lit at time } 1 \mid \xi_0(v)=1)\geq  (1- \exp(-c \lambda |S|))/e$,\\
\item[(iv)] $\pn(v \textrm{ lit during }[\exp(c \lambda^2 |S|), 2\exp(c \lambda^2 |S|)]\mid v \textrm{ lit at time } 0)\ge 1- 2 \exp (-c \lambda ^2 |S|)$.  
\end{itemize}
\end{lem}
\begin{proof}
Parts (i), (ii) and (iii) are exactly Lemma 3.1 in \cite{MVY}, and (iv) can be proved similarly, see for instance \cite{C} (similar results can be found in \cite{BBCS, CD, D,P}). 
\end{proof}

\noindent \textit{Proof of Proposition \ref{psta}.} 
We first handle the easy case when there is some $1\le i \le k$, 
such that $\deg(v_i) \geq n/2$. First by Lemma \ref{lst} we know that w.h.p. the virus survives inside $S(v_1)$ at least a time 
$\exp(c\lambda^2 d_1)$. Since by hypothesis $d_1$ diverges when $n$ tends to infinity, and since $v_1$ and $v_i$ 
are at distance at most two (both are connected to $v_0$), 
we deduce that w.h.p. $v_i$ will be infected before the extinction of the virus. 
The proposition follows by another use of Lemma \ref{lst}.

\vspace{0.2cm}
We now  assume that $d_i \leq n/2$, for all $i$. 
First we need to introduce some more notation.
For $s<t$ and $v,w \in S(v_i)$, we write 
\begin{equation}
\label{iconnect}
(v,s)\mathop{\longleftrightarrow}^{(i)} (w,t),
\end{equation}
if there exists an infection path entirely inside $S(v_i)$ joining $(v,s)$ and $(w,t)$. 
Similarly if $V$ and $W$ are two subsets of $G$, 
we write 
$$V\times \{s\} \mathop{\longleftrightarrow}^{(i)} W \times \{t\},$$
if there exists $v\in V\cap S(v_i)$ and $w\in W\cap S(v_i)$, such that \eqref{iconnect} holds. 
Now for  $\ell \geq 0$ and $1\leq i \leq k$ define
\begin{align*}
E_{\ell , i}: =  \left\{\xi_{\ell n^2} \times \{\ell n^2\} \mathop{\longleftrightarrow}^{(i)} S(v_i)\times \{(\ell +1)n^2\} \right\}.
\end{align*}
We claim that for any $\ell \geq 0$ and $1 \leq i \leq k$, we have
\begin{align} \label{ps1}
\pn\big(E_{\ell,i} \cap \left(\cap_{j \neq i} E_{\ell+1,j}^c\right) \big) \leq \exp(- c \lambda^2 n),
\end{align}
for some constant $c>0$. To fix ideas we will prove the claim for $i=1$ (clearly by symmetry there is no loss of 
generality in assuming this) and to simplify notation we also assume that $\ell =0$ (the proof works the same for any $\ell$).
Furthermore, in the whole proof the notation $c$ will stand for a positive constant independent of $\lambda$, 
whose value might change from line to line. 

\vspace{0.2cm}
\noindent Now before we start the proof we give a new definition. We denote by $(\xi'_t)_{t\ge 0}$ the contact process on $\overline S(v_1):=S(v_1)\cup \{v_0\}$, which is defined  
by using the same Poisson processes as $\xi$, but only on this subgraph. In particular with $\xi'$, the vertex $v_0$ can only be infected by $v_1$, 
and thus the restriction of $\xi$ on $\overline S(v_1)$ dominates $\xi'$. We also assume that the starting configurations of $\xi'$ and of the restriction 
of $\xi$ on $\overline S(v_1)$ are the same.   
Now for any integer $m \leq n$, define 
$$G_m=\left\{ \xi'_t(v_0)=1 \textrm{ for all } t\in [3m+2,3m+3]\right\}.$$
Let also $\kF_t=\sigma(\xi'_s,s\le t)$ be the natural filtration of the process $\xi'$. Then 
observe that for any vertex $w\in S(v_1)$, conditionally on $\kF_{3m}$, and on the event $\{ \xi'_{3m}(w)=1\}$, we have 
\begin{eqnarray*}
G_m &\subset&  \{\kN_w \cap [3m,3m+1] = \varnothing, \kN_{(w,v_1)} \cap[3m,3m+1] \neq \varnothing, \kN_{v_1} \cap[3m,3m+2] = \varnothing, \\
&&  \kN_{(v_1,v_0)} \cap[3m+1,3m+2] \neq \varnothing, \kN_{v_0} \cap[3m+1,3m+3] = \varnothing\},
\end{eqnarray*}
at least if $w\neq v_1$. Moreover, the event on the right hand side has probability equal to $(1- e^{- \lambda})^2 e^{-5}$, which is 
larger than $c\lambda^2$, for some $c>0$, and a similar result holds if $w=v_1$. Therefore for any $m$ and any nonempty subset $A \subset S(v_1)$,  
\begin{equation*}
\pn(G_m^c \mid \kF_{3m} )\, {\bf 1}(\xi'_{3m}=A) \leq (1-c\lambda^2){\bf 1}(\xi'_{3m}=A).
\end{equation*}
In other words, if we define  
$$H_m=\{\xi'_{3m} \cap S(v_1) \neq \varnothing\},$$
we get 
\begin{equation*} \label{es1}
\pn(G_m^c \mid \kF_{3m})\, {\bf 1}(H_m) \leq 1-c\lambda^2, 
\end{equation*}
for all $m\le n$. By using induction, it follows that 
\begin{align*}
\pn \left( \left(\bigcup_{m=0}^{n-1} G_m \right)^c \cap \left(\bigcap_{m=0}^{n-1} H_m\right)\right) \leq (1-c\lambda^2)^n.
\end{align*}
But by construction
$$E_{0,1}\ \subset \  \bigcap_{m=0}^{n-1} H_m.$$
Therefore 
$$\pn \left(E_{0,1}\cap\{\exists m\in [0,3n-1]\, :\,  \xi'_t(v_0)=1 \textrm{ for all }t\in [m,m+1] \}^c \right) \le \exp(- c\lambda^2 n).$$
Then by repeating the argument in each interval $[3 Mn,3(M +1)n]$, for every $M\le n/3-1$, we get  
\begin{equation}
\label{kM}
\pn \left(E_{0,1} , \, |\kM| < n/3\right) \le \exp(-c\lambda^2 n ),
\end{equation}
where 
$$\kM:= \left\{m \le n^2-1\,  :\,  \xi'_t(v_0)=1 \textrm{ for all }t\in [m,m+1] \right\}  .$$
Now for each $2\le j\le k$ and $m\le n^2-1$, define, 
\begin{eqnarray*}
&&C_{m,j}:= \left\{\kN_{(v_0,v_j)}\cap [m,m+1]\neq \varnothing,\, \kN_{v_j}\cap [m,m+2]=\varnothing\right\}\\
& \cap &\left\{|\{w\in S(v_j): \kN_{(v_j,w)}\cap[m+1,m+2]\neq \varnothing\textrm{ and }\kN_w\cap[m+1,m+2] =\varnothing\}|> \frac{\lambda d_j}{16e}\right\}.
\end{eqnarray*}
Note that these events are independent of $\kM$ and $E_{0,1}$, as they depend on different Poisson processes. 
Note also that by using \eqref{ld} 
\begin{eqnarray}
\label{Cmj}
\nonumber \pn(C_{m,j}) & =& (1-e^{-\lambda})e^{-2} \times \pn(\kB(d_j, (1-e^{-\lambda})/e)\ge \lambda d_j/(16e))\\
 &\ge & c\lambda,
\end{eqnarray}
and thus (since $C_{m,j}$ and $C_{m',j}$ are independent when $m-m'\ge 2$),  
$$\pn\left(\bigcap_{m\in \kM} C_{m,j}^c\ \Big| \ |\kM|\ge n/3\right) \le \exp(-c\lambda n).$$ 
Moreover, by construction if $m\in \kM$ and $C_{m,j}$ holds, then $v_j$ is lit at some time $t\in [m+1,m+2]$. 
Therefore by using \eqref{kM},  
\begin{eqnarray}
\label{finalprop}
\pn\left(E_{0,1}\cap\{ \exists j\in\{2,\dots,k\}:  v_j \textrm{ is never lit in }[0,n^2]\}\right) \le \exp(-c\lambda n).
\end{eqnarray}
Finally define $U_j = \exp(c\lambda^2 d_j)$, for all $j\le k$, with the constant $c$ as in Lemma \ref{lst}, 
and take $C$ large enough, so that the hypothesis $d_j\lambda^2\ge C\log n$ implies $U_j\ge 2n^2$. 
Then \eqref{finalprop} together with Lemma \ref{lst} (i) imply that 
$$\pn\left(E_{0,1} \cap (\cap_{j\ge 2} E_{1,j}^c)\right)\le \exp(-c\lambda^2n) + \prod_{j\ge 2} U_j^{-1} \le 2\exp(-c\lambda^2n/2),$$
where for the last inequality we used that $d_2+\dots +d_k \ge n/2$. This concludes the proof of \eqref{ps1}. 
The proposition immediately follows, since by using Lemma \ref{lst}, we also know that $\pn(E_{0,1})=1-o(1)$, when $v_1$ is infected initially (observe that $\exp(c\lambda^2d_1) \ge n^2$, if the constant $C$ in the hypothesis is large enough). 
\hfill $\square$

\section{Proof of Theorem \ref{td}}
\label{sectionproof}
The proof is the same in all the cases we considered, so to fix ideas we assume in all this section 
that the degree distribution is given by \eqref{pnaj} with $a\in(1,2)$. The other cases are left to the reader.

Let $(t_n)$ be as in the statement of Theorem \ref{td}. Define for $v\in V_n$, 
$$ X_{n,v}={\bf 1}(\{\xi^{v}_{t_n} \neq \varnothing\}).$$ 
The self-duality of the contact process (see (1.7) p. 35 in \cite{L}) implies that for any $\gamma >0$,
\begin{equation*} 
\mathbb{P}\left( |\xi^{V_n}_{t_n}| > \gamma  n \right) =
\mathbb{P} \left( \sum_{v\in V_n} X_{n,v} > \gamma n \right)
\end{equation*}
and similarly  for the reverse inequality. 
Hence, to prove that $|\xi^{V_n}_{t_n}|/n $ converges in probability to $ \rho_a(\lambda)$, 
we have to show that 
\begin{equation} \label{ek1}
\mathbb{P}\left( \sum_{v\in V_n} X_{n,v} > (\rho_{n,a}(\lambda) + \varepsilon) n \right) \to 0\quad \textrm{as }n\to \infty
\end{equation}
and 
\begin{equation} \label{ek2}
\mathbb{P}\left( \sum_{v\in V_n} X_{n,v} < (\rho_{n,a}(\lambda) - \varepsilon) n \right) \to 0\quad \textrm{as }n\to \infty
\end{equation}
for all $\varepsilon>0$ (recall that $\rho_{n,a}(\lambda)$  converges to $ \rho_a(\lambda),$ as $n \rightarrow \infty$). 
We will prove these two statements in the next two subsections. 

\subsection{Proof of \eqref{ek1}}
\label{subsectionYnv}
This part is quite elementary. The idea is to say that if the virus survives for a time $t_n$ 
starting from some vertex $v$, then $v$ has to infect one of its neighbors before $\sigma(v)$ 
(recall the definition \eqref{sm}), unless 
$\sigma(v) \ge t_n$, but this last event has $o(1)$ probability so we can ignore it.   
Now the probability that $v$ infects a neighbor before $\sigma(v)$, is bounded by the probability that one of the Poisson point 
processes associated to the edges emanating from $v$ has a point before $\sigma(v)$ (actually it is exactly equal to this 
if there is no loop attached to $v$). 
Then having observed that the latter event has probability exactly equal to $\rho_{n,a}(\lambda)$, 
we get the desired upper bound, at least in expectation. The true upper bound will follow using Chebyshev's inequality and the domination of the $X_{n,v}$'s by suitable i.i.d. random variables.    

\vspace{0.2cm}
\noindent Now let us write this proof more formally.  
Set $Y_{n,v} = {\bf 1}(C_{n,v})$, with (recall \eqref{sme}) 
$$C_{n,v}= \left\{  \min_{e:v\to \cdot} \sigma(e) <\sigma(v)  \right\},$$
where the notation $e:v\to \cdot$ means that $e$ is an (oriented) 
edge emanating from $v$ (possibly forming a loop).  
By construction the  $(Y_{n,v})_{v\in V_n}$ are i.i.d. random variables, and 
moreover, the above discussion shows that for all $v$, 
\begin{equation}
\label{XYnv} 
 X_{n,v} \leq Y_{n,v} + {\bf 1}(\{ \sigma(v) >t_n \}). 
\end{equation}
Now we have
\begin{align} \label{Cnv}
\en(Y_{n,v})=\mathbb{P} (C_{n,v})&= \sum_{j=1}^n \mathbb{P} (C_{n,v}\mid  D_v=j)\mathbb{P} (D_v=j)  \notag \\
& =  \sum_{j=1}^n \frac{j \lambda}{j \lambda +1} p_{n,a}(j)  = \rho_{n,a}(\lambda).
\end{align}
Therefore it follows from Chebyshev's inequality that 
\begin{equation*}
\label{Ynv}
\mathbb{P} \left( \sum_v Y_{n,v} > (\rho_{n,a}(\lambda) + \varepsilon/2) n \right)  =o(1), 
\end{equation*} 
for any fixed $\varepsilon>0$. 
On the other hand $\pn(\sigma(v) >t_n)=e^{-t_n} =o(1)$. Thus by using Markov's inequality we get
$$\mathbb{P} \left( \sum_v {\bf 1}(\{\sigma(v)>t_n \}) > \varepsilon  n/2 \right) =o(1).$$
Then \eqref{ek1}  follows with \eqref{XYnv}.

\subsection{Proof of \eqref{ek2}}
This part is more complicated and requires the results obtained so far in Sections 2, 3 and 4. 
First define $Z_{n,v}= {\bf 1}(A_{n,v}\cap B_{n,v})$, for $v\in V_n$, where
$$A_{n,v}= \{ v \textrm{ infects one of its neighbors before } \sigma(v)\},$$
and $B_{n,v}= \{\xi_{t_n}^v\neq \varnothing\}$. 
Remember that $X_{n,v} = {\bf 1}(B_{n,v})$, which in particular gives $Z_{n,v} \leq X_{n,v}$. 
Therefore the desired lower bound follows from the next lemma and Chebyshev's inequality.  
\begin{lem} \label{d4}
For any $v\neq w \in V_n$, 
\begin{itemize} 
\item[(i)] $\en (Z_{n,v} )\geq \rho_{n,a}(\lambda) - o(1)$. \\
\item[(ii)] $\cov (Z_{n,v}, Z_{n,w}) =o(1)$. 
\end{itemize}
\end{lem}
\begin{proof}
We claim that  
\begin{equation}
\label{ABnv}
\pn(B_{n,v} \mid A_{n,v}) = 1-o(1).
\end{equation}
To see this first use that w.h.p. there is a large two-step star graph in $G_n$ (given by Proposition \ref{stara<2}). 
Then use Lemma \ref{lta1} (iii) and (v) to see that w.h.p. all neighbors of $v$ have large degree and are connected to all the $v_i$'s  of the two-step star graph (recall that by construction $D_{v_i}\ge n/2$, for all $i$). 
Note that in the case $a=2$, this is not exactly true, but nevertheless the neighbors of 
$v$ and the $v_i$'s are still w.h.p. at distance at most two, 
since they are all connected to the set of vertices $z$ satisfying $D_z\ge n/\log n$ (and w.h.p. this set is nonempty).  
Now if a neighbor, say $w$, of $v$ is infected and has large degree, then Lemma \ref{lst} shows that 
w.h.p. the virus will survive in the star graph formed by $w$ and its neighbors for a long time. But if in addition $w$ and $v_1$ 
are connected (or more generally if they are at distance at most two), then $v_1$ will be infected as well w.h.p. before extinction of the process. Then Proposition \ref{psta} gives \eqref{ABnv}.

On the other hand observe that 
$$\{s_v=0\} \cap C_{n,v}\, \subset \, A_{n,v}.$$ 
Therefore \eqref{Cnv} and Lemma \ref{lta1} (iv) give Part (i) of the lemma. The second part follows easily 
by using that we also have $A_{n,v} \subset C_{n,v}$, and that the $C_{n,v}$'s are independent.   
\end{proof}

\section{Proof of Theorem \ref{propexp}}
We first prove a lower bound on the probability that the  extinction time is smaller than $n^2$.
Together with the following lemma, we will get the assertion (ii) of the theorem: 
\begin{lem}
\label{mintaun}
For every $s>0$, we have 
$$\pn(\tau_n\le s)\le \frac{s}{\en(\tau_n)}.$$
\end{lem} 
This lemma is a direct consequence of the Markov property and the attractiveness of the contact process, see for instance Lemma 4.5 in \cite{MMVY}. 

\vspace{0.2cm}
For simplicity we assume that $\lambda \le 1$, 
and leave to the reader the task to slightly modify the values of some constants in the case $\lambda>1$. 
We also assume first that the degree distribution is given by \eqref{pnaj}.

Let $\bar{n}_a$ be the number of vertices having degree larger than $n^{1/2a}$. Then $\bar{n}_a \sim \bin(n, \bar{p}_a),$ where $\bar{p}_a = \sum_{j > n^{1/2a}} p_{n,a}(j) \asymp n^{(1-a)/2a}$. Hence, as for Lemma \ref{lb}, there exists a constant $K>0$, such that 
\begin{equation*} \label{tt}
\pn\left(\bar{n}_a \le K n^{(1+a)/2a}\right) = 1-o(1). 
\end{equation*}
In fact thanks to Lemma \ref{lb}, we can even assume that 
\begin{equation}
\label{En}
\pn(\mathcal E_n) = 1-o(1),
\end{equation}
where
$$\mathcal{E}_n := \left\{n_j \le K n j^{-a} \textrm{ for all } j \le n^{1/2a}\right\} \cap \left\{\bar{n}_a \le K n^{(1+a)/2a}\right\}.$$  
Now if a vertex has degree $j$, the probability that it becomes healthy before spreading infection to another vertex is at least equal to $1/(1+ j \lambda)$ (it is in fact exactly equal to this if there is no loop attached to this vertex).  Since this happens independently for all vertices, we have that a.s. for $n$ large enough, on $\mathcal{E}_n$,
\begin{align*}
  \mathbb{P}(\tau_n \leq \min_v \sigma(v) \mid (D_v)_{v\in V_n}) & \geq (1/(1+  \lambda n))^{\bar{n}_{a}} \prod \limits_{j=1}^{ \, \,n^{1/2a}}(1/(1+  \lambda j))^{n_j} \\
& \geq  (2 \lambda n)^{-\bar{n}_{a}}  \prod \limits_{j=1}^{1/\lambda} 2^{-n_j}\prod \limits_{1/\lambda}^{\, \, n^{1/2a}} (2 \lambda j)^{-n_j} \\
& \geq (2\lambda)^{-n} n^{-\bar{n}_{a}} \prod \limits_{1/\lambda}^{\,\,n^{1/2a}} j^{-n_j}\\
& \geq \exp \left(-n \left( \log( 2 \lambda) + K \sum\limits_{1/\lambda}^{n^{1/2a}} j^{-a} \log j \right) - \bar{n}_{a} \log n \right) \\
& \geq \exp(-Cn/4),
\end{align*}
for some constant $C=C(\lambda)>0$.

Now for each vertex $v$, $\sigma(v)$ is an exponential random variable with mean $1$. Hence, a.s. for $n$ large enough and on $\mathcal{E}_n$, 
$$
 \pn(\tau_n \leq n^2 \mid (D_v)_{v\in V_n}) \geq e^{-Cn/4} - \pn(\exists v : \sigma(v) \geq n^2) \geq  e^{-Cn/2}.
$$
The same can be proved in the case when the degree distribution is given by \eqref{paj}. One just has to use that w.h.p. 
all the degrees are bounded by $n^{2/(a-1)}$, but this does not seriously affect the proof.

Together with \eqref{En}, it follows that 
$$\pn(\tau_n \leq n^2)\ge  \exp(-Cn)(1-o(1)),$$ 
and as we already mentioned above, with Lemma \ref{mintaun} we get the assertion (ii) of the theorem.

\vspace{0.2cm}
\noindent We now prove (i). This will be a consequence of a more general result: 
\begin{prop} \label{pcel}
Let $(G_n^0)$ be a sequence of connected graphs, such that $|G_n^0|\le n$, for all $n$. 
Let $\tau_n$ denote the extinction time of the contact process on $G_n^0$ 
starting from full occupancy. Assume that
\begin{align} \label{nas}
\frac{D_{n,\max}}{d_n \vee \log n} \rightarrow \infty,
\end{align}
with $D_{n,\max}$ the maximum degree and $d_n$ the diameter of $G_n^0$.  Then 
\begin{align*} 
\frac{\tau_n}{\en (\tau_n)}\quad  \mathop{\longrightarrow}^{(\kL)}_{n\to \infty} \quad  \kE(1),
\end{align*}
where $\kE(1)$ is an exponential random variable with mean one.
\end{prop} 

\begin{proof} 
According to Proposition 1.2 in \cite{M} and Lemma \ref{mintaun} above it suffices to show that there exists a sequence $(a_n)$, such that $a_n=o(\en(\tau_n))$ and 
\begin{eqnarray}
\label{xivxi}
\sup_{v\in V_n}\, \pn(\xi^v_{a_n} \neq \xi_{a_n}, \xi^v_{a_n} \neq \varnothing) = o(1),
\end{eqnarray}
where $(\xi_t)_{t\ge 0}$ denotes the process starting from full occupancy.

Set $\bar{\lambda}= \lambda \wedge 1$. Using Lemma \ref{lst}, we get 
\begin{align}
\label{taunDnmax}
\en(\tau_n) \geq \exp(c \bar{\lambda}^2 D_{n,\max}), 
\end{align}
with $c$ as in this lemma. 
Using next \eqref{nas}, we can find a sequence $(\varphi_n)$ tending to infinity, such that
\begin{align}
\label{Dnmaxdn}
\frac{D_{n,\max}}{(\log n \vee d_n)\varphi_n} \rightarrow \infty.
\end{align}
Now define 
\begin{align*}
b_n= \exp(c \bar{\lambda}^2 (\log n \vee d_n) \varphi_n ) \quad \textrm{and}\quad a_n=4b_n+1.  
\end{align*}
Then \eqref{taunDnmax} and \eqref{Dnmaxdn} show that $a_n=o(\en(\tau_n))$, so it amounts now to prove \eqref{xivxi} for this choice of $(a_n)$. To this end it is convenient to introduce the dual contact process. 
Given some positive real $t$ and 
$A$ a subset of the vertex set $V_n$ of $G_n$,  the dual process $(\hat{\xi}^{A,t}_s)_{s\le t}$ is defined by 
\[\hat{\xi}^{A,t}_s = \{ v\in V_n : (v,t-s)\longleftrightarrow A \times \{ t \} \},\]
for all $s\le t$. It follows from the graphical construction that for any $v$, 
\begin{eqnarray} \label{cl2}
&&\nonumber \pn(\xi^v_{a_n} \neq \xi_{a_n}, \xi^v_{a_n} \neq \varnothing)\\
 &=& \pn (\exists w\in V_n : \xi^v_{a_n}(w) = 0,\, \xi^v_{a_n} \neq \varnothing,\, \hat{\xi}^{w,a_n}_{a_n} \neq \varnothing) \notag\\
&\le & \sum_{w\in V_n} \pn\left(\xi^v_{a_n} \neq \varnothing,\, \hat{\xi}^{w,a_n}_{a_n} \neq \varnothing, \textrm{ and } \hat{\xi}^{w,a_n}_{a_n-t} \cap  \xi^v_t  = \varnothing \textrm{ for all } t\le a_n\right), 
\end{eqnarray}
So let us prove now that the last sum above tends to $0$ when $n\to \infty$. Set 
$$\beta_n = [\varphi_n (d_n \vee \log n)],$$
and let $u$ be a vertex with degree larger than $\beta_n$. Let then $S(u)$ be a star graph of size $\beta_n$ centered at $u$.  
Now we slightly change the definition of a lit vertex, and say that $u$ is lit if the number of its infected neighbors \textit{in $S(u)$} is larger than $\bar{ \lambda} \beta_n/(16e)$. 
We first claim that 
\begin{align}
\pp(\xi^v_{b_n} \neq \varnothing, u \textrm{ is  not lit before } b_n )  = o(1/n). \label{vbn1}
\end{align}  
To see this, define $K_n=[b_n/d_n]$ and for any $0 \leq k \leq K_n-1$
\[A_k:=\{\xi^v_{kd_n}\neq \varnothing\},\] 
and 
\[B_k:= \left\{ \xi_{kd_n}^v\times\{kd_n\} \longleftrightarrow (u,(k+1)d_n-1) \right\} \cap \{u \textrm{ is lit at time } (k+1)d_n\}.\]
 Note that 
\begin{align}
\label{inc.bn}
\{\xi^v_{b_n} \neq \varnothing, u \textrm{ is  not lit before } b_n\} \ \subset \  \bigcap_{k=0}^{K_n-1} A_k \cap B_k^c.
\end{align} 
Moreover, by using a similar argument as for \eqref{Cmj}, we obtain   
\begin{eqnarray*}
\pn\left((z,t)\longleftrightarrow (z',t+d_n-1)\right) \ge \exp(-C d_n) \quad \textrm{for any $z, z'\in V_n$ and $t\ge 0$},
\end{eqnarray*}
for some constant $C>0$ (in fact this is not true if $d_n=1$; but in this case one can just consider time intervals of length  $d_n+1$ instead of $d_n$). On the other hand,  Lemma \ref{lst} (iii) implies that  if $u$ is infected at time $t$ then it is lit at time $t+1$ with probability larger than $1/3$, if $n$ is large enough.
Therefore for any $k\le K_n-1$, 
$$\pn(B_k^c\mid \kG_k){\bf 1}(A_k) \le 1-\exp(-Cd_n)/3,$$
with $\kG_k$ the sigma-field generated by all the Poisson processes introduced in the graphical construction  in the time interval $[0, kd_n]$. Iterating this, we get  
\begin{eqnarray*}
\pn\left(\bigcap_{k=0}^{K_n-1} A_k \cap B_k^c\right) &\le  & (1-\exp(-Cd_n)/3)^{K_n-1} = o(1/n),
\end{eqnarray*}
where the last equality follows from the definition of $b_n$. Together with \eqref{inc.bn} this proves our claim \eqref{vbn1}. 
Then by using Lemma \ref{lst} (iv) we get 
\begin{align} \label{vb}
\pp(\xi^v_{b_n} \neq \varnothing, u \textrm{ is not lit at time } 2b_n ) = o(1/n).
\end{align}
Therefore, if we define 
\begin{align*}
\kA(v)&=\{\xi^v_{b_n} \neq \varnothing, u \textrm{ is lit at time $2b_n$}  \}, 
\end{align*}
we get 
$$
\pp(\kA(v)^c, \xi^v_{b_n} \neq \varnothing)=o(1/n).
$$
Likewise if 
\begin{align*}
\hat{\kA}(w)&= \{\hat{\xi}^{w,4b_n+1}_{b_n}  \neq  \varnothing,  \exists \, U \subset S(u): |U| \geq \frac{\bar{\lambda}}{16e}\beta_n\textrm{ and } (x,2b_n+1) \leftrightarrow (w,4b_n+1) \, \forall \, x \in U \}.
\end{align*}
then  
$$ 
\pp(\hat{\kA}(w)^c, \hat{\xi}^{w, 4b_n+1}_{b_n} \neq \varnothing)=o(1/n). 
$$
Moreover, $\kA(v)$ and $\hat{\kA}(w)$ are independent for all $v$, $w$. 
Now the result will follow if we can show that for any $A,B \subset S(u)$ with $|A|, |B|$ larger than $\bar{\lambda}\beta_n/(16e)$
\begin{eqnarray} \label{aub}
\pp(A \times \{2b_n\} \mathop{\longleftrightarrow}^{S(u)} B \times \{2b_n+1\} ) = 1-o(1/n),
\end{eqnarray} 
where the notation 
$$A \times \{2b_n\} \mathop{\longleftrightarrow}^{S(u)} B \times \{2b_n+1\}$$ 
means that there is an infection path inside $S(u)$ from a vertex in $A$ at time $2b_n$ to a vertex in $B$ at time $2b_n+1$. 
To prove \eqref{aub}, define
\begin{align*}
\bar{A} &=\{x \in A \setminus \{u\}: \kN_{x} \cap [2b_n,2b_n+1] = \varnothing\},\\
\bar{B} &=\{y \in B \setminus \{u\}: \kN_{y} \cap [2b_n,2b_n+1] = \varnothing\}.
\end{align*}
Since for any $x$,  
$$\pp(\kN_{x}\cap [2b_n,2b_n+1] = \varnothing)= 1-e^{-1},$$
standard large deviations results show that $|\bar{A}|$ and $|\bar B|$ are larger than $(1-e^{-1}) \bar{\lambda} \beta_n/(32e)$, with probability at least $1-o(1/n)$. Now let 
$$\kE= \{|\bar{A}| \geq (1-e^{-1}) \bar{\lambda} \beta_n/(32e)\} \cap \{|\bar{B}| \geq (1-e^{-1}) \bar{\lambda} \beta_n/(32e)\}.
$$
Set  
\begin{align*}
\varepsilon_n = \frac{1}{(\log n) \sqrt{\varphi_n}} \quad \textrm{and} \quad J_n = \left[ \frac{(\log n) \sqrt{\varphi_n}}{2}\right],
\end{align*}
and define for $0 \leq j \leq J_n-1$
\begin{eqnarray*}
C_j& =& \{\kN_u \cap [2b_n+2j \varepsilon_n ,2b_n+(2j+2) \varepsilon_n  ] = \varnothing\} \\
&\cap& \{ \exists x \in \bar{A}: \kN_{(x,u)} \cap [2b_n+2j \varepsilon_n ,2b_n+(2j+1) \varepsilon_n  ] \neq \varnothing\} \\
 &\cap&  \{ \exists y \in \bar{B}: \kN_{(u,y)} \cap [2b_n+(2j+1) \varepsilon_n ,2b_n+(2j+2) \varepsilon_n  ] \neq \varnothing\}.
\end{eqnarray*}
Observe that 
\begin{eqnarray}
\label{unionCj}
\bigcup_{j=0}^{J_n-1} C_j \subset \Big\{A \times \{2b_n\} \mathop{\longleftrightarrow}^{S(u)} B \times \{2b_n+1\}\Big\}. 
\end{eqnarray} 
Moreover, conditionally on $\bar A$ and $\bar B$, the events $(C_j)$ are independent, and
\begin{align*}
\pp(C_j \mid \bar A, \bar B) & = e^{-2 \varepsilon_n} \pp(\bin(|\bar A|, 1- e^{-\varepsilon_n}) \geq 1)\times \pp(\bin(|\bar B|, 1- e^{-\varepsilon_n}) \geq 1)\\
&\ge  1/2,
\end{align*}
on the event $\kE$, if $n$ is large enough. Therefore
\begin{eqnarray*}
\pn\left(\kE, \,  \bigcap_{j=0}^{J_n-1} C_j^c\right) &\le  & 2^{-J_n} = o(1/n).
\end{eqnarray*}
This together with \eqref{unionCj} imply  \eqref{aub}, and concludes the proof of the proposition.
\end{proof}

\begin{rema} \label{remexp}
\emph{This proposition can be used in various examples, for instance    
to the case of the configuration model with degree distribution satisfying $p(1)=p(2)=0$, and 
$$p(k) \sim c k^{-a}\qquad \textrm{as }k\to\infty,$$
for some constants $c>0$ and $a>2$. This is the degree distribution considered in \cite{CD, MMVY}. 
In this case it is known that w.h.p. the graph is connected and has diameter $\kO(\log n)$, see \cite[Lemma 1.2]{CD}, 
and since the maximal degree is at least polynomial, the proposition applies well here.  
It also applies to the preferential attachment graph model considered by Berger et al \cite{BBCS}, see \cite{C}. 
}
\end{rema}

\begin{rema}
\label{tn}
\emph{Assume that on a sequence of graphs $(G_n)$, one can prove that w.h.p. $\tau_n\ge \varphi(n)$, 
for some function $\varphi(n)$, and that in the mean time we can prove \eqref{xivxi} for some $a_n\le \varphi(n)$. 
Then observe that if \eqref{etd} holds with $t_n=a_n$, then by using the self-duality, we can see that the same holds as well  
with $t_n=\varphi(n)$. In particular, in our setting, by using Theorem \ref{propexp}, we deduce that \eqref{etd} holds with  
$t_n = \exp(c n)$, for any $c<c_{\textrm{crit}}:=\liminf (1/n)\log \en(\tau_n)$, but (using again Theorem \ref{propexp}) it does not  
when $c > c_{\textrm{crit}}$. This argument also explains why the combination of the results in \cite{MVY} and \cite{MMVY} give the statement that was mentioned in the introduction for the case $a>2$.  
}
\end{rema}

Now to complete the proof of Theorem \ref{propexp} (i), it remains to show that the hypothesis of the proposition is well satisfied in our case, namely 
for the maximal connected component -- call it $G_n^0$ -- of the configuration model $G_n$. 
It amounts to show first that the size of all the other connected components is much smaller, to ensure that 
w.h.p. the extinction time on $G_n$ and on $G_n^0$ coincide. 
Remember that with Theorem \ref{td} we know that on $G_n$ it is w.h.p. larger than $\exp(cn)$. 
In the mean time we will show that the diameter of $G_n^0$ 
is $o(n)$. Since we could not find a reference, we provide a short proof here (in fact much more is true, see below).

For $v\in V_n$, we denote by $\kC(v)$ the connected component of $G_n$ containing $v$, 
and by $||\kC(v)||$ its number of edges. 
We also define 
$$d'_n := \max_{v\notin G_n^0} \ ||\kC(v)||.$$ 

\begin{lem} \label{dmg}
Let $G_n$ be the configuration model with $n$ vertices and degree distribution given either by \eqref{pnaj} or \eqref{paj}, 
with $a\in (1,2]$. 
Let $d_n=\textrm{diam}(G_n^0)$ be the maximal distance between pair of vertices in $G_n^0$.  
Then there exists a positive constant $C$, such that w.h.p.
\begin{eqnarray*}
\max(d_n,d'_n)  \leq 
\left\{ 
\begin{array}{ll}
 C & \textrm{when } 1<a<2 \\ 
 4 \log n/ \log \log n &\textrm{when } a=2.\\  
\end{array}
\right. 
\end{eqnarray*}
\end{lem}
\begin{proof}
We only prove the result for $a=2$ here, the case $a<2$ being entirely similar. 
To fix ideas we also assume 
that the degree distribution is given by \eqref{pnaj}, but the proof works as well with \eqref{paj}. 
Set 
$$F=\left\{v: D_v\ge (\log n)^4\right\}.$$
Lemma \ref{ltb1} (iii) shows that w.h.p. all the elements of $F$ are in the same connected component, and 
Lemma \ref{ltb1} (v) then shows that w.h.p. this component has size $n(1-o(1)$, in particular it is the maximal connected component. In conclusion we get  
\begin{equation}
\label{F}
\pn(F\subset G_n^0) = 1-o(1).
\end{equation} 
Now let 
$$R_n: = \sum_{v\in F} D_v.$$ 
By construction, the probability that any stub incident to some vertex $v\notin F$ is matched with a stub incident to a vertex 
lying in $F$ is equal to $R_n/(L_n-1)$. By iterating this argument, we get 
$$\pn \left(d(v,F)>k \textrm{ and }||\kC(v)|| > k  \mid (D_w)_{w \in V_n}\right) \leq \frac{R_n}{L_n-1} \frac{R_n}{L_n-3}\cdots \frac{R_n}{L_n-2k+1},$$
for any $k$, where $d(v,F)$ denotes the graph distance between $v$ and $F$ 
(which by convention we take infinite when there is no element of $F$ in $\kC(v)$).  
Then it follows from Lemma \ref{ltb1} (i) and the fact that $R_n\asymp n\log \log n$, that 
$$\pn \left(d(v,F) > k_n  \textrm{ and }||\kC(v)|| > k_n \right) \leq \left(\frac{C \log \log n}{ \log n} \right)^{ 2\log n / \log \log n-1} = o(n^{-1}),$$
for some constant $C>0$, with $k_n=2\log n/( \log \log n) -1$. 
This proves the lemma, using a union bound and \eqref{F}.  
\end{proof}

To complete the proof of Part (i) of the theorem, we just need to remember that on any graph with $k$ edges, and for any 
$t\ge 1$, 
the extinction time is bounded by $2t$ with probability at least $1-(1-\exp(-Ck))^t$ (since on each time interval of length $1$ it has probability at least $\exp(-Ck)$ to die out, for some constant $C>0$, independently of the past). 
Therefore the previous lemma shows that w.h.p. the extinction time on $G_n^0$ and on $G_n$ are equal, 
as was announced just above the previous lemma. Then Part (i) of the theorem follows with Proposition \ref{pcel}.

\section{Extension to more general degree distributions} 
\label{secext}
We present here some rather straightforward extensions of our results to more general degree distributions.

A first one, which was also considered in \cite{VVHZ}, is to take distributions 
which interpolate between \eqref{pnaj} and \eqref{paj}: for any fixed $\alpha \in [1,\infty]$, define 
$$p_{n,a,\alpha}(j):= c_{n,a,\alpha} \, j^{-a} \qquad \textrm{for all }1\le j\le n^\alpha,$$
where $(c_{n,a,\alpha})$ are normalizing constants, 
and with the convention that the case $\alpha = \infty$ corresponds to the distribution given by \eqref{paj}.

It turns out that if $a <2$ and $\alpha < 1/(a-1)$, one can use exactly the same proof  as in the case $\alpha =1$. 
When $\alpha > 1/(a-1)$, using that w.h.p. all vertices have degree smaller than $n^{1/(a-1)} \log \log n$, 
one can use the same proof as in the case $\alpha = \infty$. The case $\alpha = 1/(a-1)$ is more complicated, 
and as in \cite{VVHZ}, a proof would require a more careful look at it.

When $a =2$, using that w.h.p. all vertices have degree smaller than $n \log \log n$, one can see that the same proof applies for any $\alpha>1$.

\vspace{0.2cm}
Another extension is to assume that there exist positive constants $c$ and $C$, and some fixed $m\ge 1$, such that for any  vertex $v$, 
$$c j^{-a} \leq \pn(D_v=j) \leq C j^{-a} \qquad \textrm{for }m \le j\le n^\alpha,$$
say with $\alpha = 1$, but it would work with $\alpha=\infty$ as well. 
The only minor change in this case is in the proof of Lemma \ref{q}. 
But one can argue as follows: just replace the set $A_1$ 
by the set of vertices in $A_m$ whose first $m-1$ stubs are not connected to any of the vertices in $E$. 
By definition these vertices have at most one neighbor in $E$ and moreover, it is not difficult to see that 
this set also has w.h.p. a size of order $n$. Then the rest of the proof applies, mutadis mutandis. All other arguments in the proof of Theorem \ref{td} remain unchanged. Therefore in this case we obtain that:
$$\frac{|\xi^{V_n}_{t_n}|}{n} - \rho_{n,a}(\lambda) \xrightarrow{ \,\,  (\pn) \, \, \, } 0,$$
with $\rho_{n,a}(\lambda)$ as in \eqref{lm}. Theorem \ref{propexp} remains also valid in this setting.  
 
\vspace{0.2cm}
\noindent \textbf{Acknowledgments:} We thank Daniel Valesin for pointing out a gap in the proof of a previous version of Proposition 6.2.

\end{document}